\documentclass[twoside, 12pt]{article}
\usepackage{amssymb, amsmath, mathrsfs, amsthm}
\usepackage{graphicx}
\usepackage{color}
\usepackage[top=1in, bottom=1in, left=1in, right=1in]{geometry}
\usepackage{float, caption, subcaption}
\usepackage{tikz}
\usetikzlibrary{shapes,arrows}
\usepackage{wrapfig}
\usepackage{float}
\usepackage{xcolor}
\usepackage{todo}
\usepackage{fix-cm}
\usepackage{setspace}
\usepackage{tikzit}
\usetikzlibrary{arrows.meta}
\usepackage[utf8]{inputenc}
\usepackage{setspace}
\usepackage{pgf,tikz}

\tikzstyle{new vertex style 0}=[fill=black, draw=black, shape=circle]
\tikzstyle{new vertex style 1}=[fill=red, draw=black, shape=circle]

\tikzstyle{new edge style 0}=[->, label]
\tikzstyle{new edge style 1}=[fill=none, scale=2, -]
\tikzstyle{new edge style 2}=[draw=red, ->]

\usetikzlibrary{backgrounds}
\usetikzlibrary{automata}
\usetikzlibrary{positioning}
\usetikzlibrary{arrows}
\usepackage{newtxmath}
\usepackage{palatino}
\usepackage{epstopdf}
\newtheorem{theorem}{Theorem}[section]
\newtheorem{example}[theorem]{Example}
\newtheorem{lemma}[theorem]{Lemma}
\newtheorem{definition}{Definition}[section]
\newtheorem{corollary}[theorem]{Corollary}

\begin{document}
\title{\bf Hosoya Polynomials of Mycielskian Graphs}
\author{Sanju Vaidya\footnote{Mercy College, 555 Dobbs Ferry, New York, {\tt SVaidya@mercy.edu}}, Aihua Li\footnote{Montclair State University, 1 Normal Avenue, Montclair, NJ 07044 {\tt lia@montclair.edu}}}
\date{}
\maketitle

\begin{abstract}
Vulnerability measures and topological indices are crucial in solving various problems such as the stability of the communication networks and development of mathematical models for chemical compounds. In 1947, Harry Wiener introduced a topological index related to molecular branching. Since then, more than 100 topological indices for graphs were introduced. Many graph polynomials play important roles in measuring such indices. Hosoya polynomial is among many of them. Introduced by Hosoya in 1988, the Hosoya polynomial of a given graph $G$ is a polynomial with the coefficients being the numbers of pairs of vertices in $G$ with all possible distances. For a given graph $G$, an extension graph is called Mycielskian graph of $G$, defined by Mycielski in 1955. In this paper, we investigate relationships between the Hosoya polynomial of any graph and that of its Mycielskian graph. The results are applied to compute the vulnerability measures, closeness and betweenness centrality, and the extended Wiener indices of selected graphs and their Mycielskian graphs. In the network science, these measures are commonly used to describe certain connectivity properties of a network. It is fascinating to see how graph polynomials are useful in other scientific fields.

\

\noindent
{\bf Keywords:} Hosoya polynomial, Wienner index, Macielskian graph, Closeness, Betweenness centrality 
\end{abstract}

\thispagestyle{empty}

\section{Introduction}

Dr. Abhyankar \cite{AS}, the thesis advisor of the first author, wrote a poem, ``Polynomials and Power series, may they forever rule the world! We cannot agree more with his statement on polynomials and power series. In this paper, we discuss a special type of  polynomial called Hosoya polynomial defined by Hosoya in 1988 \cite{HH} (also called Wiener polynomial). It is known that from Hosoya polynomials one can produce various formulas for vulnerability measures of communication networks as well as topological indices of molecular graphs of chemical compounds. 

In the study of communication networks, the vertices correspond to the processors of the networks and the edges correspond to the links among the processors.   In the last twenty years, many scientists introduced various vulnerability measures of graphs for assessing the stability and reliability of communication networks. For example, Dangalchev \cite{D} introduced a new vulnerability measure, the residual closeness, and proved that it is more sensitive than other measures of vulnerability such as graph toughness and integrity. Betweenness Centrality \cite{BBF, CG, CF} of a graph is another important vulnerability measure which is crucial in the analysis of various types of networks such as social networks. 

In a molecular graph of chemical compounds the vertices correspond to the atoms and the edges correspond to the bonds between them. A topological index (connectivity index) is a type of  molecular descriptor that is based on the molecular graph of a chemical compound. In 1947 Harry Wiener \cite{HW} introduced a topological index and correlated the indices with the boiling points of certain chemical compounds, alkanes. In 1988, Hosoya \cite{HH} introduced a polynomial (also called Wiener polynomial due to Sagan, Yeh, and Zhang \cite{SYZ})  which is based on the distances between the vertices of a graph. Many scientists and mathematicians proved relations between the derivatives of the Hosoya polynomial and the extended Wiener indices and developed various mathematical models such as Quantitative Structure-Property Relationships (QSPR) models of chemical compounds \cite{BDGG, C, DB, EIGG, KD, MR} . 

One research challenge is on the computation of the vulnerability measures and other related  topological indices for certain given graphs. In \cite{SV}, by computing Hosoya polynomials, the first author of this paper derived formulas of  vulnerability measures, residual closeness and vertex residual closeness, and a set of indices for various graphs. Those indices include topological indices, Wiener index, Hyper-Wiener index introduced by Randic \cite{MR} , TSZ index introduced by Tratch, Stankevich, and Zeffrov \cite{TSZ} , and Harary index introduced by Plavsic, Nikolic,Trinajstic, and Mihalic \cite{PNT}. 

Mycielski  introduced a fascinating method to construct a class of triangle-free graphs with large chromatic numbers from a given graph $G$ \cite{JM}. Such a graph is called  the Mycielskian graph of $G$. In this paper, the big question we want to answer is: ``What can we obtain for the Hosoya polynomial of the Mycielskian graph of a given graph $G$ from that of the $G$ itself?'' In Section 2 we  review some definitions and existing results about the vulnerability measures and topological indices used in this paper. In Sections 3, we  give a relation between the Hosoya polynomial of a given simple graph and that of its Mycielskian graph (the main result). The explicit formulas for the Hosoya polynomials of several well-known graphs are given. In section 4, we apply the results from section 3 to  establish formulas of the vulnerability measures, closeness and betweenness centrality, and extended Wiener indices of certain graphs. 
Final discussion and  conclusion are give in Section 5.

\section{Review of Graph Vulnerability Measures \\
and Topological Indices}

Throughout, the graphs we consider are all simple undirected graphs and we use the notations and terminologies introduced in Cash \cite{C}, Dangalchev \cite{D}, and Vaidya (Joshi) \cite{SV}. 
For any non-negative integers $k, n$ with $k<n$, we adopt the notation $\llbracket k, n\rrbracket=\{k, k+1, \ldots, n\}$. We first  recall  several basic concepts in graph theory which are related to distances. 

\begin{definition}
Let $G=(V(G),E(G))$ be any graph with the vertex set $V(G)=\{1, 2, \ldots , n\}$. Consider two distinct vertices $i$ and $j$ in $V(G)$.
\begin{enumerate}
    \item The distance between  $i$ and $j$ is the length of a shortest path between $i$ and $j$, denoted $d(i, j)$. The diameter of $G$, denoted $D(G)$, is the maximum of all the distances among the vertices of $G$. 
    \item Let $k$ be a vertex of $G$ and $i\ne k\ne j$.  The distance between $i$ and $j$ in the graph $G-k$, the resulting graph after removing $k$  and all the incident edges to $k$ from $G$, is denoted $d_k(i,j)$. \item Let $m$ be a non-negative integer. The number of all the pairs of vertices of $G$ with distance $m$ is denoted $d(G, m)$. 
\end{enumerate}
\end{definition}
Next we review certain vulnerability measures and topological indices of graphs. 
In order to  assess the stability of communication networks, Dangalchev introduced the following graph vulnerability measures, Closeness, and Vertex Residual Closeness \cite{D}. These concepts are popular in network science when a network is represented by a graph. 

\begin{definition} \cite{D}\label{c1}
Let $G=(V,E)$ be any graph with the vertex set $V=\{1, 2, \ldots , n\} $.
\begin{enumerate}
    \item The Closeness, $C(G)$, of $G$ is defined as 
$$C(G) =\sum_{i=1}^n C(i), \quad \mbox{where} \ \ C(i) = \sum_{j\not= i}\frac{1}{2^{d(i,j)}}.$$
Here  $C(i)$ is called the closeness of the vertex $i$.
\item The Vertex Residual Closeness (VRC) of $G$ is defined as 
$$R(G) = \mbox{min}_{k\in V(G)} (C(G-k)).$$
\end{enumerate}
\end{definition} 
Regarding the closeness of the graph $G$, a special interest is on the induced subgraph $G-k$, where $k$ is a selected vertex of $G$. The sum of such closeness defines the vertex residual closeness in (2) above. 

Four similar indices were introduced. First, the Wiener index was introduced in 1947 by Wiener \cite{HW}. In 1990, Tratch, Stankevich, and Zeffrov  introduced the TSZ index \cite{TSZ}; In 1993, Randic  introduced the Hyper-Wiener index \cite{MR}. Also in 1993,  Plavsic, Nikolic, Trinajstic, Mihalic introduced the Harary index \cite{PNT}. The definitions are summarized below.

\begin{definition}\label{c2} 
Let $G=(V(G), E(G))$ be an simple graph. 
\begin{enumerate} 
\item \cite{HW} The Wiener index of $G$ is defined as
$W(G) = \sum_{i>j} d(i,j).$
\item \cite{TSZ} The TSZ index of a graph $G$ is defined as
$$
TSZ(G) = \frac{1}{6} \sum_{i>j} [d(i,j)^3 + \frac{1}{2} \sum_{i>j} d(i,j)^2 + \frac{1}{3} \sum_{i>j} d(i, j).
$$
\item \cite{MR} The Hyper-Wiener index of $G$ is defined as 
$$WW(G) =\frac{1}{2} \sum_{i>j} \left[d(i,j)^2 + d(i, j)\right].$$
\item \cite{PNT} The Harary index of a graph G is defined as
$Har(G) = \sum_{i>j} \frac{1}{d(i, j)}.$
\end{enumerate}
\end{definition}

Hosoya introduced the following polynomial for any given  graph \cite{HH}.

\begin{definition}\label{hp}
Let $G$ be a graph with $n$ vertices and diameter $D=D(G)$. The Hosoya polynomial (also called Wiener polynomial) is defined as follows:
$$
H(G, x) = \sum_{i = 1}^D d(G, k) x^k,
$$
where $d(G, k) $ is the number of pairs of vertices in $G$ having distance $k$, as defined in Definition \ref{c1}(3). 
\end{definition}

Note that from the above definition, $d(G,1)=|E(G)|$. So the coefficient of the $x$-term of $H(G,x)$ is the number of edges of $G$. Also, in the literature, there is another version for the definition of the Hosoya polynomial, which includes the non-zero constant term $|V(G)|$, that is, 
$H(G, x) = \sum_{i = 0}^D d(G, k) x^k. $ However, in this paper,  all of our Hosoya polynomials have 0 constant term. 
Many scientists and mathematicians used the Hosoya polynomial (Wiener polynomial) and its derivatives to investigate quantitative structure-property relationships and quantitative structure-activity relationships  of chemical compounds. For example,  
Estrada -Ivanciuc- Gutman -Gutierrez -Rodriguez  introduced the $n$th order Wiener index as the $n$th derivative of the Hososya polynomial evaluated at 1 and showed their importance in organic chemistry \cite{EIGG}. Moreover, Cash \cite{C} and Bruckler-Doslic-Graovac-Gutman \cite{BDGG}  provided relationships between the derivatives of the Hosoya polynomial and the extended Wiener indices.

\begin{definition} \label{wn}
(Estrada-Ivanciuc-Gutman-Gutierrez-Rodriguez \cite{EIGG}): For any positive integer $n$, the $n$th order Wiener index is defined as follows.
$$
^nW(G) = \left[\frac{d^n H(G, x)}{dx^n}\right]_{x=1}.
$$
\end{definition}

Below we list a few known results about the previously introduced indices, which involve derivatives of the Hosoya polynomials.

\begin{theorem} \label{idn}
(Bruckler, Doslic, Graovac, Gutman \cite{BDGG}, Cash \cite{C}): Let $G$ be any simple graph  and $H(G, x)$ be the Hosoya polynomial of  $G$. Then
\begin{enumerate}
    \item [(i)]  The Wiener index $W(G)= H'(G, 1)$.
\item [(ii)] The Hyper-Wiener index is $ WW(G) = 0.5f''(1) $ where $f(x) = xH(G, x)$.
\item [(iii)]
The TSZ index $TSZ(G)=\frac{1}{6} g'''(1)$, where $g(x) = x^2H(G,x)$.
\item [(iv)] The Harary index $Har(G)=\int_0 ^1 h(x)dx$, where  $h(x) = \frac{H(G, \, x)}{x}$.
\end{enumerate}
\end{theorem}

\section{ Hosoya Polynomials}

    In this Section we establish a relation between the Hosoya polynomial of a graph and its Mycielskian graph defined in \cite{JM}. We give a formal definition below.

\begin{definition}\label{mu}
Let $G=(V(G), E(G))$ be a graph with $n$ vertices and $m$ edges. Assume $V(G)=\{v_1, v_2, \ldots ,  v_n\}$.
The Mycielskian graph of $G$, denoted $\mu(G)=(V(\mu(G)), E(\mu(G))$, has $G$ as an induced subgraph, having $n+1$ additional vertices and $2m+n$ additional edges where:   
\begin{eqnarray*}
V(\mu(G))&=& \{v_1, \ldots , v_n, u_1, u_2, \ldots, u_n, w\}\quad \mbox{and} \\
E(\mu(G))&=&E(G)\cup\{ wu_i\mid 1\le i\le n\}\cup\{u_iv_j, u_jv_i\mid \mbox{ if } v_iv_j\in E(G)\}.
\end{eqnarray*}
\end{definition}

We adopt the following popular notations in graph theory: $P_n$ for the path of length $n$ and $C_n$ for the cycle of length $n$, where $n$ is a non-negative integer.  The Mycielskian graphs $\mu(P_1)$ and $\mu(P_2)$ are described in the following example.

\begin{example}\label{mup2}
Obviously, $H(P_1,x)=x$, $H(P_2,x)=2x+x^2$, and $H(C_5,x)=5x+5x^2$
It is straighforwrd to check that $\mu(P_1)\cong C_5$. So $H(\mu(P_1),x)=5x+5x^2$.
The graph $\mu(P_2)$, build from $P_2=v_1v_2v_3$ (in red), is shown below which has 7 vertices and 9 edges. 

\ 

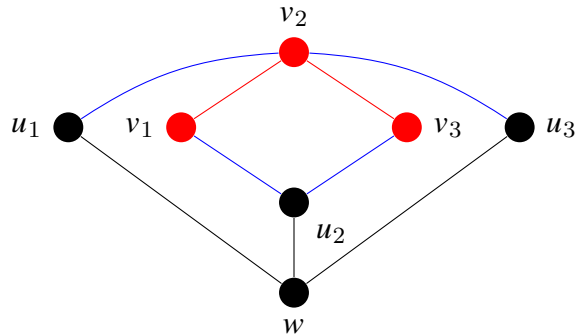
\begin{center}
\begin{tikzpicture}
	\begin{pgfonlayer}{nodelayer}
		\node [fill=red, circle, label={above:$v_2$}] (0) at (0, 1) {};
		\node [fill=black, circle, label={right:$u_3$}] (1) at (3, 0) {};
		\node [fill=black, circle, label={left:$u_1$}] (2) at (-3, 0) {};
		\node [fill=red, circle, label={left:$v_1$}] (3) at (-1.5, 0) {};
		\node [fill=red, circle, label={right:$v_3$}] (4) at (1.5, 0) {};
		\node [fill=black, circle, label={315:$u_2$}] (5) at (0, -1) {};
		\node [fill=black, circle, label={below:$w$}] (6) at (0, -2.2) {};
	\end{pgfonlayer}
	\begin{pgfonlayer}{edgelayer}
		\draw (6) to (1);
		\draw[blue] (5) to (4);
		\draw[blue] (3) to (5);
		\draw (2) to (6);
		\draw[red] (3) to (0);
		\draw (5) to (6);
		\draw[red] (0) to (4);
		\draw [bend left=15, blue] (0) to (1);
		\draw [bend left=15, blue] (2) to (0);
	\end{pgfonlayer}
\end{tikzpicture}
\end{center}

\ 

Note that the diameter of $\mu(P_2)$ is 2, $d(\mu(P_2),1)=9$ (the number of edges), and $d(\mu(P_2),2)=12$ (the number of pairs of vertices with distance 2).  Thus, we have the formula for the Hosoya polynomial:
$$H(\mu(P_2), x)=9x+12x^2 .$$
\end{example}

The degree of the Hosoya polynomial of a graph $G$ is the diameter of $D=D(G)$. In \cite{FMB}, a useful formula, which relates the diameter of $G$ and that of $\mu(G)$, is given.

\begin{theorem} \cite{FMB}
For a graph $G$ without isolated vertices, 
$$D(\mu(G))=\min\left(\max(2, D(G)), 4\right).$$
\end{theorem}

Immediately, we have 
\begin{corollary}\label{dm2}
If $G$ is a graph without isolated vertices, then $\deg(H(\mu(G))\le 4$ and 
$$D(G)=2\iff D(\mu(G))=2.$$
\end{corollary}
For any graph $G$ of diameter 2, we only need to calculate $d(G,1)$ and $d(G,2)$ to obtain $H(\mu(G),x)$. 
\begin{lemma} \label{ud2}
If $G$ is a graph with $n$ vertices and $m$ edges and the diameter of  $\mu(G)$ is 2, then 
$$H(\mu(G),x)=(3m+n)x+\left(2n^2-3m\right )x^2.$$
\end{lemma}

\proof
Note that $\mu(G)$ has $2n+1$ vertices and $3m+n$ edges by Definition \ref{mu}. Thus $d(\mu(G), 1)$   $=3m+n$ and 
$$d(\mu(G),2)=\binom{2n+1}{2}-(3m+n)=2n^2-3m .$$

A type of join graph from two given graphs has diameter 2. The definition is given below. 

\begin{definition}
Let $G_1=(V(G_1), E(G_1))$ and $G_2=(V(G_2), E(G_2))$ be two graphs. The join graph of $G_1$ and $G_2$ is denoted  as  $G_1\oplus G_2$, whose vertex set is  $V(G_1\oplus G_2)=V(G_1)\cup V(G_2)$ and edge set is $$E(G_1\oplus G_2)=E(G_1)\cup E(G_2)\cup \{v_1v_2\mid v_1\in V(G_1), v_2\in V(G_2)\}.$$
\end{definition}

The join graph $G_1\oplus G_2$ has diameter at most 2. It is because $d(u_1, u_2)=1$ for $u_1\in G_1$ and $u_2\in G_2$. For $u_1,v_1\in G_1$   and take any $v\in G_2$. Then $d(u_1, v_1)\le 2$ by the existence of the path $u_1vu_2$ of length 2. If both $G_1$ and $G_2$ are connected with diameter at least 2, then by Corollary \ref{dm2}, the diameter of $\mu(G_1\oplus G_2)$ is also 2. Then their Hosoya polynomials can be easily formulated.  

\begin{corollary} \label{joint} 
Let $G_1$ be a graph with $n_1$ vertices and $m_1$ edges and $G_2$ be a graph with $n_2$ vertices and $m_2$ edges. Assume both $G_1$ and $G_2$ are connected with diameter at least 2. Then
\begin{eqnarray*}
H(G_1\oplus G_2, x)&=&(m_1+m_2+n_1n_2)x+ \frac{1}{2}\left(n_1^2+n_2^2-n_1-n_2-2m_1-2m_2\right)x^2\quad\mbox{and}\\
H(\mu(G_1\oplus G_2),x)&=&(3m_1 + 3m_2 + 3n_1n_2 + n_1 + n_2)x+
\left(2n_1^2 + 2n_2^2 + n_1n_2 - 3m_1 - 3m_2\right )x^2.
\end{eqnarray*}
\end{corollary}

\proof
The join graph $G_1\oplus G_2$ has $n_1+n_2$ vertices and its edges are made of the edges of $G_1$, $G_2$, and $u_1u_2$, where $u_1\in G_1$ and $u_2\in G_2$. Thus, $G_1\oplus G_2$  has $m_1+m_2+n_1n_2=d(G_1\oplus G_2 , 1)$ edges and $\mu(G_1\oplus G_2)$ has $3m_1 + 3m_2 + 3n_1n_2 + n_1 + n_2$ edges by Definition \ref{mu}. Furthermore, by Lemma \ref{ud2}, 
$$d(\mu(G_1\oplus G_2, 2 )=2(n_1+n_2)^2-3(m_1+m_2+n_1n_2) 
= 2n_1^2 +2n_2^2+n_1 n_2-3m_1-3m_2 .$$ 

The star graph $S_n$ and the complete bipartite graph $K_{n,m}$ both have diameter 2 for $n\ge 2$ and $m\ge 2$.  The Hososya polynomials of their Mycielskian graphs are given below. We skip the straightforward proof.

\begin{corollary} \label{str} For $n\ge 2$ and  $m\ge 2$ , 
\begin{enumerate}
    \item[(1)]  $H(\mu(S_n),x)=(4n+1)x+(2n^2+n+2)x^2.$
    \item[(2)] 
 $H(\mu(K_{n,m}),x)=(3nm+n+m)x+(2n^2+2m^2+nm)x^2.$
\end{enumerate}
\end{corollary}
\noindent
Next, we develop the main result of the paper that relates the coefficients of the Hosoya polynomial of a graph $G$ and that of $\mu(G)$.

\begin{theorem} \label{hmy}
Let $G$ be a simple connected graph with $n$ vertices, $m$ edges, and diameter $D=D(G)$. Let   $a_i=d(G, i)$. Then the  Hosoya polynomial of $G$ can be expressed as $H(G, x) = \sum_{i=1}^D a_ix^i$. Similarly, write the Hosoya polynomial of $\mu(G)$ in the form of
$H(\mu(G),x) = \sum_{i=1}^4 b_ix^i$. Then
\begin{eqnarray*} 
b_1 &=& 3m + n, \quad b_2 = \frac{(n^2 + 3n)}{2}+3a_2,\\
b_3 &=&  n^2 - 2m - n +a_3- 2a_2, \quad  b_4 =\frac{n(n - 1)}{2} - m - a_2 - a_3.
\end{eqnarray*}
\end{theorem}

\proof  Since $\mu(G)$ has $3m+n$ edges,  $b_1 = 3m + n$. Moreover, by Corollary \ref{dm2}, $D\le 4$ and thus $\deg(H(\mu(G),x))\le 4$. 
Next we examine the pairs  of vertices with specially indicated distances. Recall that $\llbracket 1,n\rrbracket=\{1,2,\ldots , n\}$. 
\begin{enumerate}
    \item[(1)] $\forall i\in \llbracket 1,n\rrbracket, d_{\mu(G)}(w,u_i)=1$ because  $wu_i\in E(\mu(G))$.
    \item[(2)] $\forall i\in  \llbracket 1,n\rrbracket, d_{\mu(G)}(u_i,v_i)=2$. 
    
    First of all,  because $G$ is simple, $u_iv_i\not\in E(\mu(G))$. So $d_{\mu(G)}(u_i,v_i)>1$. Since $G$ is connected, $\exists k\in [1,n], k\ne i,$ such that $v_iv_k\in E(G).$ Thus $u_iv_k, u_kv_i\in E(\mu(G))$ which gives a path $u_iv_kv_i$ of length 2 in $\mu(G)$. Thus, $d_{\mu(G)}(u_i,v_i)=2.$
    \item[(3)]  $\forall i\in  \llbracket 1,n\rrbracket, $ $d_{\mu(G)}(w,v_i)=2$. With the same reason as in (2), $wu_kv_i$ is a path of length 2 from $w$ to $v_i$ and $wv_i\not \in E(\mu(G))$. Thus, $d_{\mu(G)}(w,v_i)=2.$
    \item[(4)] $\forall i,j\in \llbracket 1,n\rrbracket$ with $i\ne j$, $d_{\mu{G}}(u_i,u_j)=2$. From the definition of $\mu(G)$, $u_iu_j\not\in E(\mu(G))$ and $u_iwu_j$ is a path of length 2 in $\mu(G)$.
    \item[(5)] $\forall i,j\in \llbracket 1,n\rrbracket$ with $i\ne j$, $d_{\mu(G)}(v_i,v_j)=k\iff d_{G}(v_i,v_j)=k$ for $k=1, 2, 3$. Furthermore, $d_{\mu(G)}(v_i,v_j)=4\iff d_{G}(v_i,v_j)\ge 4$. 
 
\end{enumerate}
From all the above considerations, 
 $\forall i\in \llbracket 1,n\rrbracket$,  $ d_{\mu(G)}(v_i, u_i) = d_{\mu(G)}(w, v_i)=2$. These are  $2n$such  pairs with distance 2. For $i\ne j$, $d_{\mu(G)}(u_i, u_j ) = 2$, which gives $n(n-1)/2$ pairs of distance 2.  Furthermore, there are $a_2$ pairs of vertices in $V(G)$, say $v_i, v_j$,  satisfying $d_G(v_i, v_j ) = 2$ ($i\ne j$). Each of them produce 3 pairs of vertices in $V(\mu(G))$ with distance 2: $d_{\mu(G)}(v_i, v_j ) =  d_{\mu(G)}(v_i, u_j ) = d_{\mu(G)}(u_i, v_j )=2$. Altogether, We have
$$b_2=3a_2+2n+\frac{n(n-1)}{2}=3a_2+\frac{n^2+3n}{2}.$$
Finally, since $\mu(G)$ has $2n+1$ vertices and $(2n+1)(2n)/2=2n^2+n$ pairs of vertices, $b_1+b_2+b_3+b_4=2n^2+n$. Thus,
$$b_4=a_{\ge 4}=\frac{n(n-1)}{2}-a_1-a_2-a_3=\frac{n(n-1)}{2}-m-a_2-a_3.$$
$$b_3=2n^2+n-b_1-b_2-b_4=n^2-2m-n +a_3- 2a_2.$$.

In Example \ref{mup2}, it was shown that $H(P_2,x)=2x+x^2$ and $H(\mu(P_2),x)=9x+12x^2$. Refer to the notations in the above theorem, for $G=P_2$, we have $n=3, m=2$, $a_1=2, $ and $a_2=1$.  For $\mu(P_2)$, $b_1=9=3(2)+3$ and $b_2=12=(3^2+3(3))/2+3(1)$. It  confirms Theorem \ref{hmy}. 

\ 

In the following corollary, we derive formulas for the topological indices of Mycielskian graphs using Hosoya polynomials. 

\begin{corollary}\label{ind}
 Let $G$ be a graph with $n$ vertices and $m$ edges, and diameter $D$. Consider the  Hosoya polynomial  $H(G, x) = \sum_{i=1}^D a_ix^i$ of  $G$, where $a_i=d(G, i)$. Then 
\begin{enumerate}
\item[(1)]
$W(\mu(G)) = 6n^2 - n - 7m - 4a_2 - a_3.$
\item[(2)] $WW(\mu(G)) = 0.5(25n^2 - 11n) - 19m - 13a_2 - 4a_3$.
\item[(3)] $TSZ(\mu(G)) = 22n^2 - 8n - 22m -28a_2 - 10a_3$.
\item[(4)] $Har(\mu(G)) = \frac{1}{24}\left(17n^2 + 31n + 50m + 14a_2 + 2a_3\right)$.
\item[(5)] For $n\in \{1, 2, 3, 4\}$, $^nW(\mu(G)) =\sum_{i = 1}^ 4 i(i - 1)\cdot \cdots\cdot (i - n+1) b_i$ and $^nW(\mu(G))=0$ for $n\ge 5$.

\end{enumerate}

\proof
The results immediately follow from Theorems \ref{idn} and Theorem \ref{hmy}. 
\end{corollary}

Below we apply Theorem \ref{hmy} and Corollary \ref{ind} to the path $P_n$ (which has $n+1$ vertices) and $\mu(P_n)$ for all positive integers  $n\ge 2$. It is easy to obtain 
$H(P_n,x)=\sum_{k=1}^n (n-k+1)x^k$.

\begin{corollary}\label{path}
Let $n$ be any positive integer $\ge 2$. Then
\begin{enumerate}
\item[(1)] 
$H(\mu(P_n), x)=
(4n+1)x+\frac{1}{2}\left(n^2+11n-2\right)x^2+(n^2-2n)x^3+\frac{1}{2}\left(n^2-5n+6\right)x^4$.
\item[(2)]  (Wiener index) $W(\mu(P_n))=6n^2-n+11$.
\end{enumerate}
\end{corollary}

\proof

The results follow from the theorem \ref{hmy} and corollary \ref{ind} by setting $a_1=n$, $a_2=n-1, a_3=n-2$. 

\section{Applications - Closeness and Betweeness Centrality} 

In this section we show some applications  of Hosoya polynomials of Mycielskian graphs in describing closeness and betweenness centrality, which are crucial in the analysis of various types of networks such as communication networks and social networks. The definition of the  betweenness Centrality of a graph is as follows. 

\begin{definition} 
Let $G=(V,E)$ be a graph and $u,v,w$ are distinct vertices of $G$. Let $\sigma_{uv}(w)$ be the number of shortest paths from vertex $u$ to vertex $v$ that go through vertex $w$. Let $\sigma_{uv}$ be the total number of shortest paths from vertex $u$ to vertex $v$. 
\begin{enumerate}
    \item The {\it betweenness centrality} $B_w$  of the vertex $w$ is defined as 
    $$B_w = \sum_{u,v\not= w} b_w(u, v)\quad \mbox{with}\quad b_w(u, v) = \frac{\sigma_{uv}(w)}{\sigma_{uv}}.$$ 
    \item   The {\it Betweenness Centrality} of the graph $G$ is defined as  
    $$B^-(G) = \frac{1}{n} \sum_{w\in V} B_w.$$
\end{enumerate}
\end{definition}

Brandes, Borgatti, Freeman  proved a relation between Betweenness Centrality of a graph and the Wiener index \cite{BBF}.
\begin{theorem} \label{b-}
(Brandes, Borgatti, Freeman \cite{BBF}) Let $G$ be a graph with $n$ vertices and $W(G)$ be the Wiener index of $G$. Then the Betweenness Centrality of $G$ is given by 
$$
B^-(G) = \frac{W(G)}{n} - \frac{n-1}{2}.
$$
\end{theorem}

Turaci Okten \cite{TM} established formulas for Closeness of Mycielski graphs corresponding to various graphs. In the following theorem we give formulas for the Closeness and the betweenness centrality of the Mycielski graph $\mu(G)$ of any given  graph $G$ using the Hosoya polynomial of $G$.

\ 

\begin{theorem} \label{cbc}
Let $G$ be any simple connected graph with $n$ vertices and $m$ edges. Assume $H(G,x)=\sum_{k=1}^{D(G)} a_kx^k$. Then
\begin{enumerate}
\item[(1)] The closeness $C(\mu (G)) = \frac{1}{16}(9n^2 + 23n + 38m + 14a_2 + 2a_3)$.
\item[(2)]
The Betweenness Centrality, $B^-(\mu(G)) = \frac{4n^2 - 2n - 7m - 4a_2 - a_3}{2n + 1}$.
\end{enumerate}
\end{theorem}

\begin{proof}
Note that $C(\mu(G))=2H(\mu(G), 0.5)$. Then the results follow from the Theorems \ref{hmy}, \ref{ind}, and \ref{b-}. 
\end{proof}

The closeness and the betweeness centrality of the Mycielskian graphs of $P_n$ and $S_n$ are calculated using Theorem \ref{cbc}. 

\begin{corollary}
 \label{clo} Let $n$ be any positive integer. Then
\begin{enumerate}
    \item $C(\mu(P_n))=\frac{9n^2+95n+14}{16}$ and $B^{-}(\mu(P_n))=\frac{4n^2-6n+8}{2n+3}$;
    \item $C(\mu(S_n))=\frac{16n^2+72n+32}{16}$  and $B^{-}(\mu(S_n))=\frac{2n^2+n+2}{2n+3}$.
\end{enumerate}
\end{corollary}

\proof
The results follow from Theorem \ref{cbc} and Corollaries \ref{str} and \ref{path}. 

\ 

We end the paper with an interesting comparison of the clossness and betweeness of the graphs $\mu(P_n)$  and $\mu(S_n)$. It is natural to think that vertices in $\mu(S_n)$ are ''closer'' than those in $\mu(P_n)$ because of the ''one-center''property of $S_n$ and the stretchness of $P_n$. From Corollary \ref{clo}, we calculate the difference of the closeness. For $n\ge 2$,
$$C(\mu(S_n))-C(\mu(P_n))=\frac{7n^2-23n+18 }{16}\ge 0.$$

The difference is positive for all $n>2$ and is bigger when $n$ is larger. In other words, the closeness of $\mu(S_n)$ is bigger than that of $\mu(P_n)$ and the difference increases  when $n$ increases. It well reflects the ''stretch'' property of $P_n$ and the ''clustering'' property of $S_n$.  Amazingly, when $n=2$, the difference is 0, that is, $C(\mu(S_2))-C(\mu(P_2))=0$, meaning $\mu(S_2)$ and $\mu(P_2)$ have the same closeness. Actually, $S_2\cong P_2$. Of course their Mycielskian graphs have the same closeness.

\section{Discussion and Conclusion}

In the last twenty years, many scientists introduced vulnerability measures and topological indices of graphs to solve various problems such as the stability of communication networks and development of Quantitative Structure-Property Relationships (QSPR) models of chemical compounds. The pioneering work of Harry Wiener inspired many mathematicians and scientists to define more topological indices and to study their applications to various fields such as medicine and network science. 

Since there are many vulnerability measures and topological indices, there are open problems of finding bounds, formulas, and usefulness of them. The French mathematician Joseph Fourier (1768-1820) said, ``Mathematics compares the most diverse phenomena and discovers the secret analogies that unite them''. In this paper we showed that how Hosoya polynomials give powerful tools to compute vulnerability measures and topological indices which are crucial in two diverse fields, network communications and chemistry. This is simply amazing!

\end{document}